\documentclass[11pt]{amsart}
\usepackage[margin=3.7cm]{geometry}

\numberwithin{equation}{section}
\usepackage{amsmath,amsfonts,amsthm,amssymb,stmaryrd,color,verbatim}
\usepackage[pagebackref]{hyperref}

\usepackage[shortlabels]{enumitem}
\usepackage{amsrefs}

\newtheorem{theorem}{Theorem}[section]
\newtheorem{lemma}[theorem]{Lemma}
\newtheorem{proposition}[theorem]{Proposition}
\newtheorem{question}[theorem]{Question}
\theoremstyle{remark}
\newtheorem{remark}[theorem]{Remark}

\newcommand{\op}{\operatorname}
\newcommand{\Cay}{\op{Cay}}
\newcommand{\diam}{\op{diam}}
\newcommand{\gap}{\op{gap}}

\newcommand{\PSL}{\op{PSL}}
\newcommand{\F}{\mathbf{F}}

\renewcommand{\Re}{\op{Re}}

\newcommand\br[1]{{\left(#1\right)}}
\newcommand\floor[1]{\left\lfloor{#1}\right\rfloor}
\newcommand{\gen}[1]{\langle{#1}\rangle}

\newcommand\eps{\varepsilon}

\def\one{\mathbf{1}}
\def\kaz{\kappa}
\def\sgp{\leqslant}
\def\nsgp{\trianglelefteqslant}

\usepackage{todonotes}

\definecolor{mycolor}{rgb}{0.55,0.0,0.16}
\definecolor{myred}{rgb}{0.75,0.0,0.16}
\definecolor{mygreen}{rgb}{0.0,0.4,0.16}
\definecolor{myviolet}{rgb}{1,0,1}
\definecolor{mypink}{rgb}{0.9,0,0.5}

\hypersetup{
colorlinks=true,
linkcolor=true,
linktocpage=true,
pageanchor=true,
hyperindex=true
}

\AtBeginDocument{
\hypersetup{
linkcolor=mycolor,
urlcolor=mypink,
citecolor=mygreen
}
}

\makeatletter
\@namedef{subjclassname@2020}{%
  \textup{2020} Mathematics Subject Classification}
\makeatother

\subjclass[2020]{Primary: 20F65; Secondary: 05C25, 11L07}

\author[Sean Eberhard]{Sean Eberhard}
\address{\parbox{\linewidth}{Sean Eberhard, Mathematics Institute, University of Warwick\\
Coventry CV4\,7AL, United Kingdom \vspace{0.1cm}}}
\email{sean.eberhard@warwick.ac.uk, eberhard.math@gmail.com}

\author[Luca Sabatini]{Luca Sabatini}
\address{\parbox{\linewidth}{Luca Sabatini, Mathematics Institute, University of Warwick\\
Coventry CV4\,7AL, United Kingdom \vspace{0.1cm}}}
\email{luca.sabatini@warwick.ac.uk, sabatini.math@gmail.com}

\thanks{This work was supported by the Royal Society through a University Research Fellowship held by S.E.\ (URF\textbackslash R1\textbackslash 221185).}

\begin{document}
\title{Expanding groups with large diameter}
\maketitle

\begin{abstract}
    We study how the spectral gap and diameter of Cayley graphs depend strongly on the choice of generating set.
    We answer a question of Pyber and Szab\'o (2013)
    by exhibiting a sequence of finite groups $G_n$ with $|G_n| \to \infty$ admitting bounded generating sets $X_n,Y_n$ such that $\operatorname{Cay}(G_n,X_n)$ is an expander
    while $\operatorname{Cay}(G_n,Y_n)$ has super-polylogarithmic diameter.
    The construction uses the semidirect product $G_n = C_p^{n-1} \rtimes S_n$ with $p$ exponentially large in $n$,
    and the analysis reduces to bounding some exponential sums of permutational type.
\end{abstract}

\vspace{0.5cm}
\section{Introduction}

Let $G$ be an infinite group generated by a finite set $X$.
Several important properties of the Cayley graph $\Cay(G,X)$ do not depend on the choice of $X$,
and are actually group properties:
the growth type, the number of ends, hyperbolicity, amenability (via the F{\o}lner condition), and property (T).
The key observation behind these facts is that, for any other finite subset $Y$ of $G$,
there exists a constant $c$ such that every element in $Y$ can be written as the product of at most $c$ elements in $X$.

The situation is different for a sequence of finite groups.
Answering a question of Lubotzky--Weiss~\cite{LW93}, Alon, Lubotzky, and Wigderson~\cite{ALW01}
constructed finite groups whose Cayley graphs are expanders with respect to one bounded-size generating set but not with respect to another.
Specifically, they showed that the groups $C_2^{p+1} \rtimes \PSL_2(p)$ have this property.
In the seminal paper~\cite{Kas07}, Kassabov showed that the symmetric groups also have this property,
and subsequently Kassabov, Lubotzky, and Nikolov~\cite{KLN06} proved that almost all nonabelian finite simple groups are expanders with respect to suitable bounded-size generating sets.

However, all of the groups listed above likely have worst-case polylogarithmic diameter.
In the case of the finite simple groups, this is the content of the famous conjecture of Babai~\cite{BS92}*{Conjecture~1.7}.
This observation prompted Pyber and Szab\'o to ask the following question:

\begin{question}[Pyber \& Szab\'o~\cite{PS13}*{Question~23}]
    \label{quesPS}
    Let $(G_n)_{n \geq 1}$ be a sequence of finite groups and suppose that
    $\Cay(G_n,X_n)$ are expander graphs for some sets $(X_n)_{n \geq 1}$ of bounded cardinality.
    Is it true that for every sequence of generating sets $(Y_n)_{n \geq 1}$ we have
    $$ \diam( \Cay(G_n,Y_n) ) \> \le \> C (\log |G_n|)^C $$
    for some constant $C$?
\end{question}

A similar question was asked by the second author at the 2024 Oberwolfach meeting \emph{Growth and Expansion in Groups} \cite{BD24}*{p.~1028,  Question 5}:
there it was asked whether it is enough to assume that $\diam(\Cay(G_n, X_n))$ is polylogarithmic, i.e., whether polylogarithmic diameter is a group property.

In this note we answer both questions negatively. In fact we show that the diameter of $G$ with respect to the second generating set can be as large as $\exp(c \sqrt{\log |G|})$.

\begin{theorem} \label{thMain}
There exist constants $\delta, \eps > 0$ and a sequence of finite groups $(G_n)_{n \geq 1}$ with $|G_n| \to \infty$
with bounded-size generating sets $X_n, Y_n \subseteq G_n$ such that
\begin{enumerate}[\normalfont(a)]
    \item $\Cay(G_n,X_n)$ are $\eps$-expander graphs, and
    \item $\diam( \Cay(G_n,Y_n) ) \geq \exp(\delta \sqrt{\log|G_n|})$.
\end{enumerate}
\end{theorem}

By contrast having polynomially large diameter \emph{is} a group property,
and such groups are called \emph{almost flat}: see \cite{BT16}*{Theorem~4.1, Corollary~4.16}.
In particular, if (a) holds then $\diam(\Cay(G_n, Y_n)) \le |G_n|^{o(1)}$,
and it is an interesting open question to find the optimal bound in (b).

We now describe the construction.
Notably, there are elements in common with both Alon--Lubotzky--Wigderson and Kassabov.
From now on we suppress the subscript $n$, which will be clear from context.
The group is just
\[G \> = \> V_0 \rtimes S_n,\]
where $V_0 < V = \F_p^n$ is the deleted permutation module of dimension $n-1$,
and $p \sim e^{cn}$ is a large prime.
Note that
\[
    \> |G| \> = \> p^{n-1} n! \> = \> \exp\br{cn^2 + O(n \log n)} .
\]
The two generating sets $X$ and $Y$ both have the form $\{v\} \cup T$ where $v \in V_0$ and $T \subset S_n$.
For $Y$ we make the unimaginative choice
\[
    v = (1, -1, 0, \dots, 0), \qquad T = \{(1, 2), (1, 2, \dots, n) \}.
\]
Clearly $\diam(\Cay(G, Y)) \ge \floor{pn/2} \ge p \sim e^{cn}$, so (b) holds.

The nontrivial part of the proof is the demonstration that there is a choice of $v$ and $T$ such that $\Cay(G, X)$ has a uniform spectral gap.
For $T$ we take a bounded-size expanding generating set for $S_n$ (which exists by the result of Kassabov).
By an analysis similar to that in \cite{ALW01},
the problem of finding $v \in V_0$ reduces to the following ``permutational exponential sum'' estimate.

For $v, w \in V = \F_p^n$, let $\gen {v,w} = \sum_{i=1}^n v_i w_i$.
Let $\one = (1, \dots, 1)$ denote the all-one vector, so that $V_0 = \one^\perp$,
and let $e_p(x) := \exp(2 \pi i x / p)$.

\begin{theorem} \label{thExpSum}
    There exist constants $\delta, \eps >0$ such that the following holds.
    For every integer $n \ge 1$ and prime $p \le e^{\delta n}$,
    there exists $v \in V_0$ such that, for all nonconstant $w \in V$,
    \[
       \left| \frac1{n!} \sum_{\sigma \in S_n} e_p\br{\gen{v, w^\sigma}} \right|
        \> \le \> 1 - \eps.
    \]
\end{theorem}

In fact we show that a random $v \in V_0$ works.
When $p$ is bounded, this can be proved by a straightforward modification of the analysis in \cite{ALW01}.
The main novelty in our result is that $p$ is allowed to be enormous:
in technical terms, the reason we can take $p$ exponential in $n$ is that we do not need a union bound over all vectors,
but only over support-one vectors, and we can then use a deterministic switching argument based on the Cauchy--Schwarz inequality to deal with all other vectors.

\begin{remark}
    We have chosen to present the simplest argument that answers Question~\ref{quesPS}, but
    a more elaborate analysis shows that Theorems~\ref{thMain} and \ref{thExpSum} are true for any constant $\delta > 0$ and some suitable $\eps = \eps(\delta) > 0$.
    Moreover, $\eps(\delta)$ can be taken to be a continuous monotonic function such that $\eps(\delta) \to 1$ as $\delta \to 0$.
    On the other hand, a standard ``rectification'' argument in additive combinatorics (see \cite{BLR98}*{Section~3} for example) shows that necessarily $\eps(\delta) \to 0$ as $\delta \to \infty$ in Theorem~\ref{thExpSum}.
\end{remark}

\section{Preliminaries} \label{sec2}

If $\Gamma$ is a $d$-regular graph,
then $\gap(\Gamma)$ denotes the spectral gap of the normalized adjacency matrix.
It is well-known that $\gap(\Gamma) > 0$ if and only if $\Gamma$ is connected,
and we say $\Gamma$ is an \emph{($\eps$-)expander} if $\gap(\Gamma) \ge \eps$ (often $\eps$ is implicit).
For an expander graph, the diameter is logarithmic in the number of vertices.

If $G$ is a finite group and $S \subseteq G$, the Cayley graph $\Cay(G, S)$ is the undirected graph with vertex set $G$ and edges $(g, gs^{\pm1})$ for $g \in G$ and $s \in S$.
We allow multiple edges, so $\Cay(G, S)$ is always a $d$-regular graph for $d = 2|S|$.
We write $\gap(G, S)$ for $\gap(\Cay(G, S))$, and if $\Cay(G, S)$ is an expander then we say that $G$ is an expander \emph{with respect to $S$}.

In practice, it is often technically convenient to argue in terms of the \emph{Kazhdan constant}.
This is defined by
\[
    \kaz(G, S) \> := \>
    \inf_{\substack{\pi \colon G \to U(V) \\ V^G = 0}}
    \inf_{\substack{\xi \in V \\ \|\xi\| = 1}}
    \max_{s \in S} \|\pi(s) \xi - \xi\| ,
\]
where the first infimum runs over all unitary representations $\pi \colon G \to U(V)$ with no nonzero $G$-invariants,
and the second infimum runs over all unit vectors $\xi \in V$.
The following basic properties of $\kaz(G,S)$ are mostly immediate from the definition.

\begin{lemma}
    \label{lem:kaz-basics}
    Let $G$ be a finite group and let $S, T \subseteq G$.
    \begin{enumerate}
        \item $\kaz(G, S) \le \kaz(G, T)$ if $S \subseteq T$;
        \item $\kaz(G, S) \le 2$;
        \item $\kaz(G, G) \ge \sqrt 2$;
        \item $\kaz(G, S) \ge \frac1n \kaz(G, S^n)$.
    \end{enumerate}
\end{lemma}
\begin{proof}
    (1)--(2) are clear. For (3)--(4) see \cite{Kas07}*{Propositions~1.3--1.4}.
\end{proof}

It is well known that the spectral gap and the Kazhdan constant are related by the following inequalities.

\begin{lemma}
    \label{lemKazSG}
    Let $G$ be a finite group and $S \subseteq G$.
    Then
    \[
        \frac{\kaz(G,S)^2}{2|S|}
        \> \le \> \gap(G,S)
        \> \le \> \frac{\kaz(G,S)^2}{2}.
    \]
\end{lemma}
\begin{proof}
    See \cite{dHRV93}*{Proposition~III}.
\end{proof}

In particular, if $S$ has bounded cardinality,
then $\gap(G, S)$ is bounded away from zero if and only if $\kaz(G, S)$ is so.

\vspace{0.1cm}
\section{Expansion in semidirect products}

Alon--Lubotzky--Wigderson~\cite{ALW01} analyze the expansion of semidirect products by identifying them as a special case of zig-zag products. Let us take the opportunity to explain a more direct approach based on Kazhdan constants.

\subsection{Kazhdan constant and semidirect products}

Let $\pi \colon G \to U(V)$ be a unitary $G$-representation, $S \subseteq G$, and $\eps > 0$.
We say that $\xi \in V$ is \emph{$\eps$-almost $S$-invariant} if
\(
    \| \pi(s)\xi - \xi \| \> < \> \eps \| \xi \|
\)
for all $s \in S$.
Then $\kaz(G, S)$ is the largest constant $\eps \ge 0$ with the property that every unitary $G$-representation with an $\eps$-almost $S$-invariant vector has a nonzero $G$-invariant vector.

\begin{lemma} \label{lemSean1}
    Let $\pi \colon G \to U(V)$ and let $\xi \in V$ be $\eps$-almost $S$-invariant.
    Assume $\kaz(G, S) > 0$ and let $\eps' = \eps / \kaz(G, S)$.
    Then
    \begin{enumerate}[(i)]
        \item there is a $G$-invariant vector $\xi_1 \in V$ such that $\|\xi - \xi_1\| < \eps'\|\xi\|$;
        \item $\xi$ is $2\eps'$-almost $G$-invariant.
    \end{enumerate}
\end{lemma}
\begin{proof}
    \emph{(i)}
    Orthogonally decompose $\xi = \xi_1 + \xi_2$ where $\xi_1 \in V^G$ and $\xi_2 \in (V^G)^\perp$.
    Since $((V^G)^\perp)^G = 0$, there is some $s \in S$ such that
    \[
        \kaz(G, S) \|\xi_2\| \le \|\pi(s) \xi_2 - \xi_2\| = \|\pi(s) \xi - \xi\| < \eps \|\xi\|.
    \]
    Hence $\|\xi_2\| < \eps' \|\xi\|$.

    \emph{(ii)}
    For $g \in G$, $\|\pi(g) \xi - \xi\| = \|\pi(g) \xi_2 - \xi_2\| \le 2 \|\xi_2\| < 2\eps' \|\xi\|.$
\end{proof}

\begin{lemma} \label{lemSean2}
    If $H \sgp G$ then
    \(
        \kaz(G, S) \ge \frac12 \kaz(G, S \cup H) \, \kaz(H, S \cap H).
    \)
\end{lemma}
\begin{proof}
    Let $\pi \colon G \to U(V)$ be a unitary representation with $V^G = 0$ and suppose $\xi \in V$ is $\eps$-almost $S$-invariant vector $\xi$, where $\eps > 0$.
    Then in particular $\xi$ is $\eps$-almost $(S \cap H)$-invariant.
    By Lemma~\ref{lemSean1}(b), $\xi$ is $2\eps'$-almost $H$-invariant, where $\eps' = \eps / \kaz(H, S \cap H)$.
    Since $\kaz(H, S \cap H) \le 2$, $\eps \le 2\eps'$, so $\xi$ is $2\eps'$-almost $(S \cup H)$-invariant.
    Since $V^G = 0$, it follows that $2\eps' \ge \kaz(G, S \cup H)$, i.e., $\eps \ge \frac12 \kaz(G, S \cup H) \kaz(H, S \cap H)$.
\end{proof}

\begin{lemma} \label{lemSean3}
    If $N \nsgp G$, then
    \(
        \kaz(G, S \cup N) \ge \frac14 \kaz(G/N, SN/N).
    \)
\end{lemma}
\begin{proof}
    Let $\pi \colon G \to U(V)$ be a unitary representation with $V^G = 0$ and suppose $\xi \in V$ is $\eps$-almost $(S \cup N)$-invariant,
    where $\eps = \frac14 \kaz(G/N, SN/N) \le 1/2$.
    By Lemma~\ref{lemSean1}(a), there is some $\xi_1 \in V^N$ such that $\|\xi - \xi_1\| < \eps' \|\xi\|$, where $\eps' = \eps / \kaz(N, N) \le \eps / \sqrt 2$.
    This implies $\|\xi_1\| > (1 - \eps') \|\xi\|$.
    Now $V^N$ is a $G/N$-representation, and if $s \in S$ then
     \begin{align*}
        \|\pi(s) \xi_1 - \xi_1\|
        &\le \|\pi(s) \xi - \xi\| + 2 \|\xi - \xi_1\|  \\
        &<  (\eps + 2\eps') \|\xi\|
        <  \frac{\eps + 2\eps'}{1 - \eps'} \|\xi_1\|
        <  4\eps \|\xi_1\| .
    \end{align*}
    Therefore $\xi_1$ is $4\eps$-almost $(SN/N)$-invariant.
    Since $(V^N)^{G/N} = V^G = 0$ and $4 \eps = \kaz(G/N, SN/N)$, this is a contradiction.
\end{proof}

The following result can be used to replace \cite{RVW02}*{Theorem 6.4} in the
context of Cayley graphs.

\begin{proposition}\label{prop:semidirect}
    Let $G = N \rtimes H$, $S \subseteq N$ and $T \subseteq H$.
    Let $R = S \cup T \subseteq G$.
    Then
    \[
    \kaz(G, R)
    \> \ge \>
    \frac{\sqrt 2}{48} \, \kaz(N , S^H) \, \kaz(H , T),
    \]
    where $S^H = \{s^h : s \in S, h \in H\}$.
\end{proposition}

\begin{proof}
    By Lemma~\ref{lemSean2} applied to $H \sgp G$ we have
    \[
        \kaz(G, R) \geq \frac12 \kaz(G, S \cup H) \kaz(H, T).
    \]
    Now $S^H \cup H \subseteq (S \cup H)^3$, so by Lemma~\ref{lem:kaz-basics}(1,4) we have
    \[
        \kaz(G, S \cup H) \ge \frac13 \kaz(G, (S \cup H)^3) \ge \frac13 \kaz(G, S^H \cup H).
    \]
    Applying Lemmas~\ref{lemSean2} and \ref{lemSean3} to $N \nsgp G$,
    \begin{align*}
        \kaz(G, S^H \cup H)
        &\ge \frac12 \kaz(G, N \cup H) \kaz(N, S^H) \\
        &\ge \frac18 \kaz(G/N, G/N) \kaz(N, S^H) \ge \frac{\sqrt2}8 \kaz(N, S^H).
    \end{align*}
    This concludes the proof.
\end{proof}

\subsection{Reduction to the exponential sum estimate}

We now prove Theorem~\ref{thMain}(a) assuming Theorem~\ref{thExpSum} (which will be proved in the next section).
Let $V= \F_p^n$ and let $V_0 < V$ be the deleted permutation module of dimension $n-1$.
Let $G = V_0 \rtimes S_n$ and $v \in V_0$.
By Kassabov~\cite{Kas07} there is a bounded-size generating set $T \subseteq S_n$ such that $\kaz(S_n, T)$ is bounded away from zero.
If $X = \{v\} \cup T$,
then by Proposition~\ref{prop:semidirect} and Lemma~\ref{lemKazSG} we have
\begin{align*}
        \kaz(G, X)
        &\ge \frac{\sqrt 2}{48} \kaz(V_0, v^{S_n}) \kaz(S_n, T)
         \\
        &\ge \frac{1}{24} \gap(V_0, v^{S_n})^{1/2} \kaz(S_n, T).
    \end{align*}
We now focus on $\gap(V_0, v^{S_n})$.
Observe that $\Cay(V,v^{S_n})$ is the disjoint union of $p$ graphs all isomorphic to $\Cay(V_0,v^{S_n})$.
It follows that the eigenvalues of the adjacency operator of $\Cay(V_0, v^{S_n})$ are the same as for
$\Cay(V, v^{S_n})$, the only difference being their multiplicities.
Since $V$ is abelian, its regular representation is isomorphic to a direct sum of one-dimensional representations
indexed by the irreducible complex characters.
We can identify these characters with vectors $w \in V$, using the formula
\[
    \chi_w(u) = e_p(\gen{u, w}) \qquad (u,w \in V).
\]
(Recall that $e_p(x) := \exp(2 \pi i x / p)$.)
It follows that the eigenvalues of $\Cay(V,v^{S_n})$ and so of $\Cay(V_0,v^{S_n})$ are given by
\[
    \frac1{2|S_n|} \sum_{\sigma \in S_n} \br{\chi_w(v^\sigma) + \chi_w(-v^{\sigma})} \hspace{1cm} (w \in V) .
\]
These are precisely the real values $\{ \Re[\lambda_{v,w}] \}_{w \in V}$, where
\[
    \lambda_{v,w} \> := \>
    \frac1{n!} \sum_{\sigma \in S_n} e_p (\gen{v, w^\sigma}) .
\]
The trivial eigenvalues are associated to the constant vectors, and so
$\gap(V_0,v^{S_n}) = 1 - \max_{w \in V \setminus \gen \one} \Re[\lambda_{v,w}]$.
 Since $\Re[\lambda_{v,w}] \le |\lambda_{v,w}|$, the reduction follows.

\vspace{0.1cm}
\section{Permutational exponential sums}

In this section we use a probabilistic argument to prove Theorem~\ref{thExpSum}, which asserts that there are constants $\delta, \eps > 0$ with the following property. For every integer $n > 0$ and prime $p \le e^{\delta n}$ there is some vector $v \in V_0$ such that
\[
    \max_{w \in V \setminus \gen\one} |\lambda_{v, w}|
    \> \le \> 1 - \eps.
\]
The main challenge in proving this is that, if $p$ is large, there are far too many vectors $w \in V \setminus \gen \one$ for a na\"ive union bound to work (even if we just consider sorted vectors).

For $v = (v_1,\ldots,v_n) \in V_0$ and $u \in \F_p$,
let $\lambda_v(u) = \lambda_{v,(u,0,\ldots,0)}$ and observe that
\[
    \lambda_v(u) \> = \> \frac1n \sum_{i=1}^n e_p( u v_i).
\]

\begin{lemma}
    \label{lem:tail-bound}
    Let $\eps \ge 2/n$ and assume $0 \ne u \in \F_p$.
    If $v \in V_0$ is uniformly random, then
    \[
        \Pr(|\lambda_v(u)| \ge \varepsilon) \> \le \>
        4 \exp(-\varepsilon^2 n / 8).
    \]
    As a consequence,
    \[
        \Pr\br{\max_{0 \ne u \in \F_p} |\lambda_v(u)| \ge \varepsilon}
        \> \le \> 4 p \exp\br{-\varepsilon^2 n / 8}.
    \]
\end{lemma}
\begin{proof}
   The random variables $(e_p(u v_i))_{i=1}^{n-1}$ are independent uniformly random complex $p$th roots of unity.
    If \(\Sigma = \sum_{i=1}^{n-1} e_p(uv_i)\),
    by Hoeffding's inequality \cite{Hoe63}*{Theorem~2} we have
    \[
        \Pr\br{|\Sigma| \ge \eps n / 2} \le 4 \exp(-\eps^2 n^2 / 8 (n-1)) \le 4 \exp(-\eps^2 n / 8).
    \]
    On the other hand if $|\Sigma| \le \eps n/2$ then by the triangle inequality
    \[
        |\lambda_v(u)|
        \le \eps/2 + 1/n \le \eps.
    \]
    The second bound follows by the union bound.
\end{proof}

We now deduce a deterministic bound for all nonconstant $w \in V$.

\begin{lemma}
    \label{lem:switching}
    Let $v \in V_0$. We have
    \[
       \max_{w \in V \setminus \gen\one} |\lambda_{v, w}|^2
       \> \le \>
       \frac12 + \frac12 \max_{0 \ne u \in \F_p} |\lambda_v(u)|^2.
    \]
\end{lemma}
\begin{proof}
    We apply a switching argument based on the Cauchy--Schwarz inequality.

    Let $w \in V \setminus \gen\one$ and let us bound $\lambda_{v,w}$.
    Obviously $\lambda_{v, w} = \lambda_{v, w^\sigma}$ for all $\sigma \in S_n$, so we may assume $w_1 \ne w_2$ without loss of generality. Let $u = w_1 - w_2$ and let $\tau$ be the transposition $(1,2)$. Then
    \[
        \lambda_{v,w}
        = \frac1{n!} \sum_{\sigma \in S_n} e_p( \gen{v, w^\sigma})
        = \frac1{n!} \sum_{\sigma \in S_n} \frac12 \br{e_p(\gen{v, w^\sigma}) + e_p(\gen{v, w^{\tau\sigma}})}.
    \]
    Applying the Cauchy--Schwarz inequality, we obtain
    \begin{align*}
        |\lambda_{v,w}|^2
        &\le \frac1{n!} \sum_{\sigma \in S_n}
        \frac14 \left|
            e_p(\gen{v, w^\sigma}) + e_p(\gen{v, w^{\tau\sigma}})
        \right|^2 \\
        &= \frac1{n!} \sum_{\sigma \in S_n}
        \frac14 \br{
            2 + 2 \Re\left[e_p(\gen{v, w^\sigma} - \gen{v, w^{\tau\sigma}})\right]
        }\\
        &= \frac12 + \frac12 \Re\left[\frac1{n!}\sum_{\sigma \in S_n}
            e_p(\langle v, (w - w^\tau)^\sigma\rangle)
        \right].
    \end{align*}
    Now $w - w^\tau = (u, -u, 0, 0, \dots)$, so
    \[
        \frac1{n!} \sum_{\sigma \in S_n} e_p(\gen{v, (w-w^\tau)^\sigma})
        = \frac1{n(n-1)} \sum_{\substack{1 \le i, j \le n \\ i \ne j}}
        e_p(u v_i - u v_j),
    \]
    and
    \[
        \sum_{\substack{1 \le i, j \le n \\ i \ne j}} e_p(u v_i - u v_j)
        = \sum_{1 \le i, j \le n} e_p(u v_i - u v_j) - n
        = n^2 |\lambda_v(u)|^2 - n,
    \]
    so we get
    \[
        |\lambda_{v,w}|^2
        \le \frac12 + \frac12 \frac{n|\lambda_v(u)|^2 - 1}{n-1}
        \le \frac12 + \frac12 |\lambda_v(u)|^2.
    \]
    This proves the lemma since $u \ne 0$.
\end{proof}

\begin{proof}[Proof of Theorem~\ref{thExpSum}]
    Let $\eps = 1/2$ and $\delta = 1/100$.
    Let $n$ be an integer, $p \le e^{\delta n}$ a prime,
    and note that $\eps \ge 2/n$ holds.
    For $v \in V_0$ uniformly random,
    by Lemmas~\ref{lem:tail-bound} and \ref{lem:switching} we have
    \begin{align*}
        \Pr\br{\max_{w \in V \setminus \gen\one} |\lambda_{v,w}|^2 \ge (1 + \eps^2)/2}
        &\le \Pr\br{\max_{0 \ne u \in \F_p} |\lambda_v(u)| \ge \eps} \\
        &\le 4 p e^{-\eps^2 n / 8}
        \le e^{(3 \delta - \eps^2 / 8) n} < 1.
    \end{align*}
    This proves the theorem with $\sqrt{(1 + \eps^2) / 2}$ in place of $1 - \eps$, which is sufficient.
\end{proof}

\vspace{0.1cm}
\bibliography{refs-exp-diam}
\vspace{0.3cm}

\end{document}